\documentclass[12pt,A4]{amsart}

\DeclareSymbolFont{AMSb}{U}{msb}{m}{n}
\DeclareSymbolFontAlphabet{\Bbb}{AMSb}

\usepackage[ansinew]{inputenc} % accents i altres
\usepackage{latexsym}

\usepackage{amstext,amsfonts,amsmath,amsthm,graphicx,amssymb,amscd,epsfig}

\newtheorem{theorem}{Theorem}[section]

\newtheorem{lemma}[theorem]{Lemma}
\newtheorem{corollary}[theorem]{Corollary}
\newtheorem{proposition}[theorem]{Proposition}

\newcommand{\C}{\mathbb{C}}
\newcommand{\R}{\mathbb{R}}
\newcommand{\Z}{\mathbb{Z}}
\newcommand{\N}{\mathbb{N}}

\newcommand{\inv}{\mathcal{E}}
\newcommand{\Equiv}{\stackrel{\rightarrow}{\mathcal{E}}}
\newcommand{\equivM}{\stackrel{\leftrightarrow}{\mathcal{E}}}

\newcommand{\tr}{\mbox{tr}}

\DeclareMathOperator{\arccot}{arccot}

\newcommand{\ee}{\textrm{e}}

\begin{document}

\title[Local attractor for a $\Z_n$ symmetric map]{\bf A local but not global attractor for a $\Z_n$-symmetric map}
\author[Alarc\'on, Castro,  Labouriau]{B. Alarc\'on\and  S.B.S.D. Castro  \and I.S. Labouriau}

\address{B. Alarc\'on --- CMUP; Rua do Campo Alegre, 687; 4169-007 Porto; Portugal ---
permanent address:
Departament of Mathematics. University of Oviedo; Calvo Sotelo s/n; 33007 Oviedo; Spain}
\email{alarconbegona@uniovi.es}
\address{ S.B.S.D. Castro --- A CMUP and FEP.UP; Rua Dr. Roberto Frias; 4200-464 Porto; Portugal}
\email{sdcastro@fep.up.pt}
\address{ I.S. Labouriau ---  CMUP and FCUP; Rua do Campo Alegre, 687; 4169-007 Porto; Portugal}
\email{islabour@fc.up.pt}

\maketitle

\begin{abstract}

There are many tools for studying local dynamics.
An important problem is how this information can be used to obtain global information.
We present examples for which local stability does not carry on globally. To this purpose  we construct, for any natural $n\ge 2$,  planar maps whose symmetry group is $\Z_n$  having a local attractor that is not a global attractor. The construction starts from an example with symmetry group $\Z_4$. We show that although this example has codimension $3$ as a $\Z_4$-symmetric map-germ, its relevant dynamic properties are shared by two $1$-parameter families in its universal unfolding.
The same construction can be applied to obtain examples that are also dissipative.
The symmetry of these  maps forces them to  have  rational rotation numbers.

\end{abstract}

\section{Introduction}

At the end of the $19^{th}$ century, Lyapunov \cite{L} related the local stability of an equilibrium point to the eigenvalues of the Jacobian matrix of the vector field at that point. This led to the Markus-Yamabe Conjecture \cite{MY} in the 1960's, and fifteen years later to a version for maps of the original conjecture, using the relation between stability of fixed points and the eigenvalues of the Jacobian matrix of the map at that point \cite{LS}. In the 1990's, this was named, by analogy, the Discrete Markus-Yamabe Conjecture and remains unproven. It may be stated as follows:

{\sc Discrete Markus-Yamabe Conjecture: }
Let $f$ be a $C^1$ map from $\R^m$ to itself such that $f(0)=0$. If all the eigenvalues of the Jacobian matrix at every point have modulus less than one, then the origin is a global attractor.

It is known that the original conjecture holds for $m=2$ and is,
in this case, equivalent to the injectivity of the vector field
\cite{G}, \cite{F}. It is false for $m>2$ \cite{BL}, \cite{CEGHM}.
On the other hand, the Discrete Markus-Yamabe Conjecture holds,
for all $m$, if the Jacobian matrix of the map is triangular and,
additionally for $m=2$, for polynomial maps \cite{Cima-Manosa}. It
is false in higher dimensions, also for polynomial maps
\cite{CEGHM}. There exists a counter-example for $m=2$ that is an injective rational map (\cite{Cima-Manosa}).
This striking difference between the discrete and
continuous versions encouraged the study of the dynamics of
continuous and injective maps of the plane that satisfy the
hypotheses of the Discrete Markus-Yamabe Conjecture. This is now
known as the Discrete Markus-Yamabe Problem. From the results in
\cite{Alarcon}, it follows that the Discrete Markus-Yamabe Problem
is true for $m=2$ for dissipative maps, by introducing as an extra
condition the existence of an invariant ray (a continuous curve
without self-intersections connecting the origin to infinity). An
invariant ray can be, for instance  an axis of symmetry.

In the presence of symmetry, that is, when the map is equivariant, the ultimate question can be stated as follows:

{\sc Equivariant Discrete Markus-Yamabe Problem: }
Let $f:\R^2\longrightarrow\R^2$ be a dissipative $C^1$ equivariant planar map such that $f(0)=0$. Assume that all eigenvalues of the Jacobian matrix at every point have modulus less than one. Is the origin a global attractor?

Given the results in Alarc\'on {\em et al.} \cite{Alarcon}, the Equivariant Discrete Markus-Yamabe Problem is true if the group of symmetries of $f$ contains a reflection. In this case, the fixed-point space of the reflection plays the role of the invariant ray. This situation is addressed in Alarc\'on {\em et al.} \cite{ACL_global}. In the present paper, we are concerned with symmetry groups that do not contain a reflection.

The Equivariant Discrete Markus-Yamabe Problem has a negative answer if the
reflection is not a group element. In fact, the example
constructed by Szlenk and reported in \cite{Cima-Manosa} satisfies
all the hypotheses of the Discrete Markus-Yamabe Problem,   is
equivariant (as we show here) under the standard action of $\Z_4$,
but the origin is not a global attractor. Indeed, there is an
orbit of period $4$ and the rotation number defined in \cite{ORN} is $\frac{1}{4}$.
The example has a singularity at the origin with $\Z_4$ codimension 3,
and we show that two inequivalent 1-parameter families in its unfolding share
these dynamic properties.

We use Szlenk's example to construct  differentiable maps on the plane with symmetry group $\Z_n$ for all $n\ge 2$.
Each example has an attracting fixed point at the origin and
 a periodic orbit of minimal period $n$ which prevents local dynamics to extend globally.
 The construction may be extended to one of the 1-parameter families mentioned above.

We adapt  $\Z_n$ symmetric example to make it dissipative. In that case its symmetry implies that the rotation number is rational. 
Implications of this fact are discussed in the final section.

\subsection{Equivariant Planar Maps}

The reference for the folllowing definitions and results is Golubitsky {\em et al.} \cite[chapter XII]{golu2}, to which we refer the reader interested in further detail.

Our concern is about groups acting linearly on  $\R^2$ and more particularly about the action of $\Z_n$, $n  \geq 2$ on $\R^2$.
Identifying $\R^2 \simeq \C$, the finite group $\Z_n$ is  generated by one element $R_n$, the rotation by $2\pi/n$ around the origin, with action given by
$$
R_n\cdot z= \ee^{2\pi i/n} z .
$$

A map $f: \R^2 \rightarrow \R^2$ is $\Z_n$-equivariant if
$$
f(\gamma x)=\gamma f(x) \;\;\; \forall \; \gamma \in \Z_n, \; x \in \R^2.
$$
We also say, if the above only holds for elements in $\Z_n$, that $\Z_n$ is the symmetry group of $f$.

Since most of our results depend on the existence of a unique fixed point for $f$, the following is a useful result.
\begin{lemma}
If $f$ is $\Z_n$-equivariant then $f(0)=0$.
\end{lemma}
\begin{proof}
We have $f(0)=f(\gamma 0)=\gamma f(0)$, by equivariance. The element $\gamma =\exp{2\pi i/n}$ of $\Z_n$ is such that $\gamma x \neq x$ for all $x\neq 0$. It then follows that $f(0)=0$.
\end{proof}

\section{Example with an orbit of period $4$}

In this section, we explore the properties of an example of a local attractor which is not global since it has an orbit of period $4$. This example is due to Szlenk and is reported in \cite{Cima-Manosa}. A list of properties for this example is given in Proposition \ref{Szlenk}. We divide this section in two subsections, the first dealing with dynamic properties and the second concerned with the study of the singularity in Szlenk's map.

\subsection{Dynamics}

Before introducing the example it is useful to establish some concepts that will be used in the proofs to come.
Let $S_{1,n}\subset\R^2$ be the open sector
$$
S_{1,n}=\left\{(x,y)=(r\cos\theta,r\sin\theta):\ 0<\theta<2\pi/n\right\}
$$
and define $S_{j,n}$, $j=2,\cdots,n$  recursively by
$S_{j,n}=R_n\left(S_{j-1,n}\right)$. Then
$\R^2=\bigcup_{j=1}^n\overline{S_{j,n}}$, where $\overline{A}$ is
the closure of $A$. Moreover, $S_{1,n}=R_n\left(S_{n,n}\right)$.
Then each $\overline{S_{j,n}}$ is a {\em fundamental domain} for the action of $\Z_n$, in particular
if $f:\R^2\longrightarrow\R^2$ is $\Z_n$-equivariant then $f$ is completely determined by its restriction to
$\overline{S_{j,n}}$.

 A {\em line ray} is a half line through the origin, of the form
$\{t(\alpha ,\beta ) :\quad t\ge 0\}$, with
$0\ne(\alpha,\beta)\in\R^2$.

The next Proposition establishes the relevant properties of Szlenk's example that will be used in the construction of other $\Z_n$-equivariant maps in the next section.

\begin{proposition}[Szlenk's example]\label{Szlenk}
Let $F_4:\R^2\longrightarrow\R^2$ be defined by
$$
F_4(x,y) = \left(-\frac{ky^3}{1+x^2+y^2}, \frac{kx^3}{1+x^2+y^2}\right)
\qquad
\mbox{for}\quad
1<k<\frac{2}{\sqrt{3}} .
$$
The map $F_4$ has  the following properties:
\begin{enumerate}
\renewcommand{\labelenumi}{{\theenumi})}
\item\label{HomeoSz}
$F_4$ is of class $C^1$.
\item\label{F4homeo}
$F_4$ is a homeomorphism.
\item\label{PtosFixosSzl}
$Fix(F_4)=\{0\}$.
\item\label{OriginalSz}
$F_4^4(P)=P$ for $P=\left((k-1)^{-1/2},0\right)$,
 with $F_4^j(P)=R_4^j(P)\ne P$ for $j=2,3$.
\item\label{LocAttSz}
$0$ is a local  attractor.
\item\label{EquivariantSz}
$F_4$ is $\Z_4$-equivariant.
\item\label{RaysInSzl}
The restriction of $F_4$ to any line ray is a homeomorphism onto another line ray.
\item\label{SzMapsSectors}
$F_4\left(\overline{S_{j,4}}\right)=\overline{S_{j+1,4}}$ for $j=1,\cdots,4\pmod{4}$ with
$F_4\left(\partial {S_{j,4}}\right)=\partial{S_{j+1,4}}$.
\item\label{astroSz}
The curve $F_4(\cos\theta,\sin\theta)$ goes across each line ray and is transverse to line rays at all points $\theta\ne \frac{m\pi}{2}$ for $m=0,1,2,3$.
\end{enumerate}
\end{proposition}

\begin{proof}
Some of the statements follow from previously established results. Since we deal with these first, the order of the proof does not follow the numbering in the list above.

Statements \ref{HomeoSz})  and  \ref{OriginalSz}) are immediate from the expressions of $F_4$ and of $P$, as remarked in \cite{Cima-Manosa}. Note that the periodic orbit of $P$ of statement \ref{OriginalSz})  lies in the boundary of the sectors $\bigcup_j\partial S_{j,4}$. 
 \medskip
 
In the appendix of  \cite{Cima-Manosa} it is shown that the eigenvalues of  $DF_4(x,y)$ lie in the open unit disk, establishing  \ref{LocAttSz}).  Statement \ref{PtosFixosSzl}) follows as a direct consequence of Corollary 2 in \cite{Alarcon-orbitas-periodicas} and the same estimates on the eigenvalues.
 \medskip

Concerning \ref{EquivariantSz}) note that $R_4$,
the generator of $\Z_4$, acts on the plane
as $R_4 (x,y) = (-y,x)$.
In order to prove that $F_4(x,y)$
 is $\Z_4$-equivariant we compute
$$
F_4(R_4 (x,y)) = (-\frac{kx^3}{1+x^2+y^2}, -\frac{ky^3}{1+x^2+y^2})
$$
and
$$
R_4 F_4(x,y) = R_4 (-\frac{ky^3}{1+x^2+y^2}, \frac{kx^3}{1+x^2+y^2}) =
(\frac{-kx^3}{1+x^2+y^2},\frac{-ky^3}{1+x^2+y^2}).
$$
Observing that these are equal establishes statement \ref{EquivariantSz}).
 \medskip
 
The behaviour of $F_4$ on line rays described  in  \ref{RaysInSzl}) is easier to understand if we write
$(x,y)$ in polar coordinates $(x,y)=(r\cos\theta,r\sin\theta)$ yielding:
\begin{equation}\label{SzlenkPolar}
F_4(r\cos\theta,r\sin\theta)=\frac{kr^3}{1+r^2}\left(-\sin^3\theta,\cos^3\theta \right)\ .
\end{equation}

From this expression it follows that for each fixed $\theta$, the
line ray through $(\cos\theta,\sin\theta)$ is mapped into the
line ray through $(-\sin^3\theta,\cos^3\theta)$. The mapping is a
bijection, since $r^3/(1+r^2)$ is a monotonically increasing
bijection from $[0,+\infty)$ onto itself. In particular, it follows from this that $F_4$ is injective and that $F_4(\R^2)=\R^2$. Since every continuous and injective map in $\R^2$ is open (see Ortega \cite[Chapter 3, Lemma 2]{OrtegaBook}), it follows that $F_4$ is a homeomorphism, establishing \ref{F4homeo}).
\medskip

The behaviour  of $F_4$ on sectors and their boundary is the essence of  \ref{SzMapsSectors}).
From the definition of the sectors we have
 $$
 S_{j+1,4}=R_4\left(S_{j,4}\right)
 $$
 and therefore, by $\Z_4$-equivariance,
 $$
 F_4\left(S_{j+1,4}\right)= F_4\left(R_4\left(S_{j,4}\right)\right)= R_4\left(F_4\left(S_{j,4}\right)\right) \ .
 $$
It then suffices to show that
$F_4\left(\overline{S_{1,4}}\right)=\overline{S_{2,4}}$.
The sectors $S_{1,4}$ and $S_{2,4}$ have the simple forms
$$
S_{1,4}=\left\{ (x,y):\quad x>0,\quad y>0\right\}
\qquad
S_{2,4}=\left\{ (x,y):\quad x<0,\quad y>0\right\} .
$$

From the expression of $F_4$ it is immediate that if $x>0$ and
$y>0$ then the first coordinate of $F_4(x,y)$ is negative and the
second is positive and thus $F_4\left({S_{1,4}}\right)\subset
{S_{2,4}}$. It remains to show the equality, which we delay until after the proof of \ref{astroSz}).

\begin{figure}
\begin{center}
\includegraphics[scale=.50] {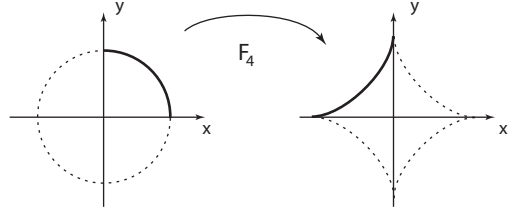}
\end{center}
\caption{Szlenk's example $F_4$ maps a quarter of the unit circle into a quarter of the  astroid
$\frac{k}{2}(-\sin^3\theta,\cos^3\theta)$.\label{Astroide}}
\end{figure}
\medskip

The expression (\ref{SzlenkPolar}) in polar coordinates shows that the circle
$(\cos\theta,\sin\theta)$, $0\le\theta\le 2\pi$ is mapped by $F_4$ into the curve
$\gamma(\theta)=\frac{k}{2}(-\sin^3\theta,\cos^3\theta)$  known as the {\em astroid}
(Figure~\ref{Astroide}).
The arc $\gamma(\theta)$, $0\le\theta\le\pi/2$  joins $(0,\frac{k}{2})$ to $(-\frac{k}{2},0)$.
Since for $\theta\in(0,\pi/2)$ the functions $\cos^3\theta$ and  $-\sin^3\theta$ are both monotonically decreasing with strictly negative derivatives,
then the  $0\le\theta\le\pi/2$ arc of the astroid has no self intersections
and the restriction of $F_4$ to the quarter of a circle $0\le\theta\le\pi/2$ is a bijection into this arc
(Figure~\ref{Astroide}).
\medskip

Moreover,  the determinant of the matrix with rows $\gamma(\theta)$ and $\gamma^\prime(\theta)$
is
$$
\det\left(\begin{array}{c}\gamma(\theta)\\ \gamma^\prime(\theta)
\end{array}\right)
=\frac{3k^2}{4}\sin^2\theta\cos^2\theta
$$
showing that the arc of the astroid is transverse at each point $\gamma(\theta)$, $0<\theta<\pi/2$ to the line ray through it. Transversality fails at the end points of the arc, but the line rays still go across the astroid at the cusp points --- this is assertion~\ref{astroSz}).
\medskip

Thus, $F_4$ induces a bijection between line rays in $S_{1,4}$ and line rays in $S_{2,4}$ and
using the radial property  \ref{RaysInSzl})
it follows that $F_4\left({S_{1,4}}\right)={S_{2,4}}$.
The behaviour on the boundary of ${S_{1,4}}$ also follows either from the radial property or from a simple direct calculation, concluding the proof of \ref{SzMapsSectors}).
\end{proof}

\subsection {Universal unfolding of $F_4$}

In this section we discuss a universal unfolding of the singularity $F_4$ in the context of $\Z_4$-equivariant maps that fix the origin  under contact equivalence. All the preliminaries concerning equivariant unfolding theory, as well as the proof of the result, are deferred to an appendix. The trusting reader may proceed without reading it.

\begin{proposition}
A $\Z_4$ universal unfolding under contact equivalence of the germ at the origin of the singularity $F_4$   is given by
$$
G_4(x,y,\alpha, \beta, \delta) = F_4(x,y) + \alpha(x,y) + \left[\beta + \delta (x^2+y^2)\right](-y,x),
$$
where parameters $\alpha$, $\beta$ and $\delta$ are real.
\end{proposition}

From the point of view of the dynamics, it is important to describe the maps in the unfolding 
that preserve the dynamic properties of $F_4$. The first result is immediate from the expression of the derivative of $G_4$ at the origin:

\begin{lemma}\label{lemaatrator}
The origin is a hyperbolic local attractor for $G_4(x,y,\alpha, \beta, \delta)$ if and only if $\alpha^2+\beta^2<1$.
\end{lemma}

Although the unfolding above refers to the germ at the origin, we show below that
its expression defines a map that shares some dynamic properties of $F_4$ for some parameter values.
These values lie on two lines in parameter space.

\begin{proposition} \label{unfoldingG4}
Let $g(x,y)$ be either $G_4(x,y,\alpha,0,0)$ or $G_4(x,y,0,\beta,0)$.
Then for $\alpha$ or $\beta$ positive and small enough,
\begin{itemize}
	\item  $g$ is a global diffeomorphism;
	\item at every point in $\R^2$  the eigenvalues of the jacobian of $g$ have modulus less than one;
	\item  there exists $p \in \R^2$ such that $g^4 (p)=p$.
\end{itemize}
\end{proposition}

\begin{proof}
The case $\alpha > 0$ is the one adressed in \cite[Theorem E]{Cima-Manosa}. We treat the case $\beta > 0$ in a similar manner.

The matrix $DF_4(x,y)$ is given in the appendix. In this proof denote it by
$$
DF_4=\left( \begin{array}{cc}
a & b \\ 
c & d
\end{array} \right).
$$
If $\mu$ is an eigenvalue of $Dg$ then
$$
\mu = \frac{1}{2}\left( -\tr{(DF_4)} \pm \sqrt{\tr^2{(DF_4)} -4\det{(DF_4)}-4\beta(\beta+c-b)} \right).
$$
We know from \cite[Theorem D]{Cima-Manosa} that all eigenvalues of $DF_4$ are zero on the coordinate axes and complex otherwise. Furthermore, all eigenvalues of $DF_4$ have modulus less than $k\sqrt{3}/2<1$. The latter statement ensures that, for any $k$ and for small $\beta$, the eigenvalues of $Dg$ also have modulus less than one. 

We want to show that all eigenvalues of $Dg$ are non-zero. When the eigenvalues of $DF_4$ are zero it is clear that those of $Dg$ are not. Away from the axes, the eigenvalues of $DF_4$ are non-zero and $\det{(DF_4)}>0$. Since $\det{(Dg)}=\det{(DF_4)}+\beta^2-\beta(b-c)$, the eigenvalues of $Dg$ are zero if and only if
$$
\det{(DF_4)} + \beta^2 = \beta (b-c).
$$
Since $b-c<0$, then for $\beta > 0$, it is always the case that the eigenvalues of $Dg$ are nonzero.

So far, we have shown that $g$ is a local diffeomorphism at every point. In order to show that it is a global diffeomorphism, we show as in \cite[Theorem E]{Cima-Manosa} that
$$
\lim_{|(x,y)| \rightarrow \infty} |g(x,y)| = \infty.
$$
This implies that $g$ is proper and we may invoke Hadamard's theorem (quoted in \cite{Cima-Manosa}) that asserts that a proper local diffeomorphism is a global diffeomorphism.

In order to establish the limit above we use polar coordinates and write
$$
g(r, \theta) = \frac{kr^3}{1+r^2}(-\sin^3{\theta}, \cos^3{\theta})+\beta(-r\sin{\theta},r\cos{\theta})
$$
and hence,
$$
|g(r,\theta)|^2 = \frac{k^2r^6}{(1+r^2)^2}(\sin^6{\theta}+\cos^6{\theta})+\beta^2 r^2 +2\beta k \frac{r^4}{1+r^2}(\sin^4{\theta}+\cos^4{\theta}).
$$
Noting now that $\sin^6{\theta}+\cos^6{\theta} \geq 1/4$ and  $\sin^4{\theta}+\cos^4{\theta} \geq 1/2$, we use $1+r^2<2r^2$ for $r>1$ to write
$$
|g(r,\theta)|^2 \geq \frac{kr^2}{16} + \beta^2 r^2 + \frac{\beta k r^2}{2} 
 \stackrel{r\rightarrow \infty}{\longrightarrow} \infty .
$$
The existence of points of period $4$ follows from the hyperbolicity of the period $4$ points of $F_4$.
\end{proof}

\section{Construction of $\Z_n$-equivariant examples}

The next examples refer to a local attractor, examples with a local repellor may be obtained considering $f^{-1}$.

\begin{theorem} \label{exper} For each $n\ge 2$ there exists  $f:\R^2\to\R^2$ such
that:
\begin{enumerate}
\renewcommand{\theenumi}{\alph{enumi}}
\renewcommand{\labelenumi}{{\theenumi})}
\item\label{C1a} $f$ is a differentiable homeomorphism;
\item\label{C1b} $f$  has symmetry group $\Z_n$;
\item \label{C1E} $Fix(f)=\{0\}$;
\item\label{C1d} The origin is a local attractor;
\item \label{C3E} There exists a periodic orbit of minimal period $n$.
\end{enumerate}
\end{theorem}

\begin{proof}
For $n\ge 2$, the map
\begin{equation}\label{DefHn}
h_n\left(r\cos\theta,r\sin\theta\right)=\left(r\cos\frac{4\theta}{n},r\sin\frac{4\theta}{n}\right)
\end{equation}
is a local diffeomorphism at all points in $\R^2\backslash \{0\}$, is continuous at $0$ and
$h_n(S_{1,4})=S_{1,n}$, $h_n(S_{2,4})=S_{2,n}$ with $\left|h_n(x,y)\right| =\left|(x,y)\right|$.
Moreover,   the restriction of $h_n$ to ${\overline{S_{1,4}}}$ is a bijection onto
$\overline{S_{1,n}}$ and $h_n$ maps each line ray through the origin into another line ray through the origin.
\medskip

Similar properties hold for the inverse
$$
h_n^{-1}\left(r\cos\theta,r\sin\theta\right)=\left(r\cos\frac{n\theta}{4},r\sin\frac{n\theta}{4}\right)
$$
with $h_n^{-1}(S_{1,n})=S_{1,4}$.

\begin{figure}
\begin{center}
\includegraphics[scale=.50] {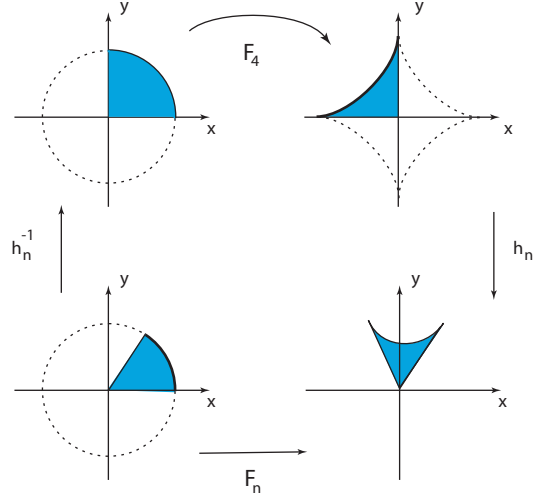}
\end{center}
\caption{Construction of the $\Z_n$-equivariant example $F_n$ in a fundamental domain of the
$\Z_n$-action, shown here for $n=6$.\label{FigSlzenkZn}}
\end{figure}
\bigskip

Let $F_n:\overline{S_{1,n}}\longrightarrow \overline{S_{2,n}}$ be defined by
(see Figure~\ref{FigSlzenkZn})
\begin{equation}\label{defineF}
F_n(x,y)=h_n\circ F_4\circ h_n^{-1}(x,y) \ .
\end{equation}

We extend $F_n$ to a $\Z_n$-equivariant map $F_n:\R^2\longrightarrow\R^2$ recursively, as follows.
\medskip

Suppose   for $1\le j\le n-1$ the map
$F_n$ is already defined in $S_{j,n}$ with $F_n(S_{j,n})=S_{j+1,n}$.
If $(x,y)\in S_{j+1,n}$ we have $R_n^{-1}(x,y)\in S_{j,n}$  and thus
$F_n\circ R_n^{-1}(x,y)$ is well defined, with
$F_n\circ R_n^{-1}(x,y)\in S_{j+1,n}$.
Define $F_n(x,y)$ for $(x,y)\in S_{j+1,n}$ as
$F_n(x,y)=R_n\circ F_n\circ R_n^{-1}(x,y)\in S_{j+2,n}$.
Finally,  for $(x,y)\in S_{n-1,n}$ we obtain $F_n(x,y)\in S_{1,n}$.
\medskip

The following properties of $F_n$ now hold by construction, using Proposition~\ref{Szlenk}:
\begin{itemize}
\item
$F_n$ is $\Z_n$-equivariant.
\item
$Fix(F_n)=\{0\}$.
\item
The origin is a local attractor.
\item
$F_n^n(P)=P$ for $P=\left((k-1)^{-1/2},0\right)$,
 with $F_n^j(P)\ne P$ for $j=2,\ldots,n-1$.
 Note that all $F_n^j(P)$ lie on the boundaries $\partial S_{j,n}$ of the sectors $S_{j,n}$.
\item
$F_n$ maps each line ray through the origin onto another line ray through the origin.
\end{itemize}

\begin{figure}
\begin{center}
\includegraphics[scale=.30] {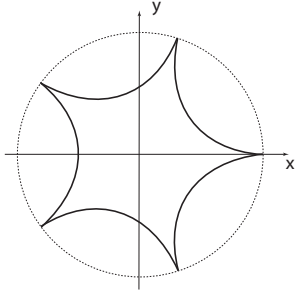}
\end{center}
\caption{Image of the circle $(\sin\theta, \cos\theta)$ by  the $\Z_n$-equivariant example $F_n$, shown here for $n=5$.\label{FigAstro5}}
\end{figure}

Since $h_n$ maps line rays to line rays, to see that $F_n$ is a homeomorphism
 it is sufficient to observe that
$\gamma_n(\theta)=F_n(\cos\theta,\sin\theta)$, $0\le\theta\le 2\pi$ is a simple closed curve that meets each line ray only once and does not go through the origin (Figure~\ref{FigAstro5}).
This is true
because  away from the origin both $h_n$ and $h_n^{-1}$ are differentiable with non-singular derivatives.
Since  $h_n$ and $h_n^{-1}$ map line rays into line rays,
it follows from assertion~\ref{astroSz}) of Proposition~\ref{Szlenk} that
$\gamma_n$ is transverse to line rays except at the cusp points $\gamma_n(\theta)$,
$\theta=\frac{2 m \pi}{n}$, $m=0,1,\ldots,n-1$
where the line ray goes  across it.

It remains to show
 that $F_n$ is everywhere differentiable in $\R^2$.
 This is done
 in Lemma~\ref{Fnsmooth}  below.
\end{proof}

\begin{lemma}\label{Fnsmooth}
$F_n$ is everywhere differentiable in $\R^2$.
\end{lemma}

\begin{proof}
First we
show that $DF_4(0,0)=(0)$ (zero matrix) implies that $F_n$ is differentiable at the origin with $DF_n(0,0)=(0)$.
That  $DF_4(0,0)=(0)$ means that for every $\varepsilon>0$ there is a
$\delta>0$ such that, for every $X\in\R^2$,  if $\left| X\right|<\delta$ then
$$
\left| F_4(X)- F_4(0,0)-DF_4(0,0) X\right|=\left| F_4(X)\right|<\varepsilon \left| X\right|\ .
$$
Since $ h_n$ and $h_n^{-1}$ preserve the norm, we have that if $Y=h_n(X)$ then $|Y|=|X|$
and furthermore, for any $Y$ such that $|Y|<\delta$ we obtain
$$
\left| F_n(Y)\right|=\left| h_n\left( F_4\left(h_n^{-1}(Y)\right)\right)\right|=
\left| h_n\left( F_4(X)\right)\right|=\left|  F_4(X)\right|
<\varepsilon \left| X\right|= \varepsilon \left| Y\right|
$$
%proving our claim, since $F_n(0,0)=(0,0)$. Therefore, since this holds for any $\varepsilon$,
Therefore, since $F_n(0,0)=(0,0)$ and since this holds for any $\varepsilon$,
$$
\lim_{|X| \rightarrow 0} \frac{|F_n(X)-F_n(0,0)-(0)X|}{|X|} = 0
$$
proving our claim. 
\medbreak

Recall that in \eqref{defineF} and in the text thereafter the map $F_n$ is made up by gluing different functions on sectors:
in $S_{1,n}$ the expression of $F_n$ is given by  $h_n\circ F_4\circ h_n^{-1}$
and in $S_{2,n}$ by $R_n\circ h_n\circ F_4\circ h_n^{-1}\circ R_n^{-1}$.
Both expressions define differentiable functions away from the origin since both $h_n$ and $h_n^{-1}$ are of class $C^1$ in $\R^2\backslash\{(0,0)\}$.
We have already shown that $F_n$ is differentiable at the origin.
It remains to prove that the derivatives of the two functions coincide at the common boundary of
$\partial S_{1,n}$ and $\partial S_{2,n}$.
At the remaining boundaries the result follows from the $\Z_n$-equivariance of $F_n$.

Since we are working away from the origin, we may use polar coordinates.
The  expressions for $h_n$, $R_n$ and their inverses take the simple forms below,
where we use $\widehat{f}$ to indicate the expression of $f$ using polar coordinates in both source and target:
$$
\widehat{h}_n(r,\theta)=\left(r,\frac{4\theta}{n}\right)
\qquad
\widehat{h}_n^{-1}(r,\theta)=\left(r,\frac{n\theta}{4}\right)
$$
$$
\widehat{R}_n(r,\theta)=\left(r,\theta+\frac{2\pi}{n}\right)
\qquad
\widehat{R}_n^{-1}(r,\theta)=\left(r,\theta-\frac{2\pi}{n}\right) \ .
$$

Let $\widehat{F}_4(r,\theta)=\left(\Psi_4(r,\theta),\Phi_4(r,\theta)\right)$ be the expression of $F_4$ in polar coordinates.
From (\ref{SzlenkPolar}) we get:
\begin{equation}\label{RPolar}
\Psi_4(r,\theta)=\frac{kr^3}{1+r^2}\sqrt{\cos^6\theta+\sin^6\theta}=
\frac{kr^3}{1+r^2}\sqrt{1-3\cos^2\theta+3\cos^4\theta}
\end{equation}
\begin{equation}\label{ThetaPolar}
\Phi_4(r,\theta)=\left\{\begin{array}{ll}
\displaystyle\arctan\left( -\frac{\cos^3\theta}{\sin^3\theta}\right)&\quad\mbox{if}\quad \theta\ne k\pi\\
\phantom{.}&\\
\displaystyle\arccot\left( -\frac{\sin^3\theta}{\cos^3\theta}\right)&\quad\mbox{if}\quad \theta\ne \frac{\pi}{2}+k\pi \ .
\end{array}\right.
\end{equation}
The derivative $D\widehat{F}_4(r,\theta)$ of $\widehat{F}_4$ is thus,
\begin{equation}\label{derivadaPolares}
\left(\begin{array}{cc}
\displaystyle kr^2\frac{3+r^2}{(1+r^2)^2}\sqrt{\cos^6\theta+\sin^6\theta}&
\displaystyle \frac{kr^3}{1+r^2}
\frac{3\sin\theta\cos\theta\left(\sin^4\theta-\cos^4\theta\right)}{\sqrt{\cos^6\theta+\sin^6\theta}}\\
\phantom{.}&\\
0&\displaystyle
\frac{3\sin^2\theta\cos^2\theta}{{\cos^6\theta+\sin^6\theta}}
\end{array}\right)
\end{equation}
where the two alternative forms for $\Phi_4(r,\theta)$ yield the same expression for the derivative.

Note that the Jacobian matrix of $\widehat{h}_n$ is constant and the same is true for its inverse.
The derivatives of both  $\widehat{R}_n$ and of $\widehat{R}_n^{-1}$ are the identity.
Let $(r,2\pi/n)$ be the polar coordinates of a point $\xi$ in
$\left(\partial S_{1,n}\cap \partial S_{2,n}\right)\backslash \{0\}$.
In order to show that the derivatives  at $\xi$ of
$\widehat{h}_n\circ \widehat{F}_4\circ \widehat{h}_n^{-1}$
and of
$\widehat{R}_n\circ \widehat{h}_n\circ \widehat{F}_4\circ \widehat{h}_n^{-1}\circ \widehat{R}_n^{-1}$
coincide, we only need to show that
$D\widehat{F}_4$ at $\widehat{h}_n^{-1}(r,2\pi/n)=(r,\pi/2)$
equals
$D\widehat{F}_4$ at $\widehat{h}_n^{-1}(\widehat{R}_n^{-1}(r,2\pi/n))=(r,0)$.
More precisely, for any $(r,\theta)$
$$
D
\widehat{h}_n(r,\theta)=A_n=\left(\begin{array}{ll}
1&0\\
0&\frac{4}{n}
\end{array}\right)
\qquad
D\widehat{h}_n^{-1}(r,\theta)=B_n=
\left(\begin{array}{ll}
1&0\\
0&\frac{n}{4}
\end{array}\right)
$$
and thus
\begin{eqnarray*}
{}&{}&
D\left(\widehat{R}_n\circ \widehat{h}_n\circ \widehat{F}_4\circ \widehat{h}_n^{-1}\circ \widehat{R}_n^{-1}\right)(\xi)\\
{}&=&D\widehat{R}_n(\widehat{h}_n(\widehat{F}_4((r,0)))
D\widehat{h}_n(\widehat{F}_4((r,0))
 D\widehat{F}_4(r,0)
D \widehat{h}_n^{-1}(r,0)
D\widehat{R}_n^{-1}(r,2\pi/n)
\\{}&=&
Id\cdot A_n\cdot
 D\widehat{F}_4(r,0)
\cdot B_n\cdot Id
\\{}&=&
A_n\cdot
 D\widehat{F}_4(r,0)
\cdot B_n
\end{eqnarray*}
and
\begin{eqnarray*}
{}&{}&
D\left( \widehat{h}_n\circ \widehat{F}_4\circ \widehat{h}_n^{-1}\right)(\xi)
\\{}&=&
D\widehat{h}_n(\widehat{F}_4((r,\pi/2))
 D\widehat{F}_4(r,\pi/2)
D \widehat{h}_n^{-1}(r,2\pi/n)
\\ {}&=&
A_n\cdot
 D\widehat{F}_4(r,\pi/2)
\cdot B_n\ .
\end{eqnarray*}

From (\ref{derivadaPolares}) it follows that
$$
D\widehat{F}_4(r,\pi/2)=D\widehat{F}_4(r,0)=
\left(\begin{array}{cc}
\displaystyle kr^2\frac{3+r^2}{(1+r^2)^2}&0\\
\phantom{.}&\\
0&0
\end{array}\right)
$$
completing our proof.
\end{proof}

The construction in the proof of Theorem~\ref{exper} only works because Szlenk's example $F_4$ has the special properties  \ref{RaysInSzl}),   \ref{SzMapsSectors})  and ~\ref{astroSz}) of
Proposition~\ref{Szlenk}. For instance, identifying $\R^2\sim\C$ the map
$f(z)=\overline{z}^3$ is $\Z_4$-equivariant, but does not have the properties above and
$h_5\circ f\circ h_5^{-1}(z)=f(z)$.
\bigbreak

Alarc\'on {\sl et al.}~\cite[Theorem 4.4]{Alarcon} construct, starting from $F_4$, an example having the additional property that $\infty$ is a repelllor.
The new example, $H(x,y)$, is of the form
$$
H(x,y)=\phi(|F_4(x,y)|) F_4(x,y)
$$
where $\phi:[0,\infty)\longrightarrow [0,\infty)$ is described in  \cite[Lemma 4.6]{Alarcon}.

Then $H$ has all the properties of Proposition~\ref{Szlenk}.
Therefore, applying to $H$ the construction of Theorem~\ref{exper} we obtain the following:

\begin{corollary}\label{SzlenkDissipative}
For each $n\ge 2$ there exists a map  $f:\R^2\to\R^2$ satisfying properties \ref{C1a})--\ref{C3E}) of
Theorem~\ref{exper} and, moreover, for which $\infty$ is a repellor.
\end{corollary}

\section{Final comments}

It remains an interesting question to find out whether our construction can be applied to $G_4$ to produce a $\Z_n$ universal unfolding of $F_n$. A partial answer is given next. The proof is straightforward.

\begin{lemma}
If $\alpha = 0$ then $G_4$ has the property that $G_4\left({S_{1,4}}\right)={S_{2,4}}$.
\end{lemma}
As a consequence, the previous construction applied to $G_4$ with $\alpha=0$ produces other examples with $\Z_n$-symmetry and period $n$ orbits. Furthermore, using Proposition \ref{unfoldingG4}, if also $\delta =0$ these new examples are diffeomorphisms.

Note that, even though the unfolding applies only locally, the dynamic properties are robust beyond this constraint as they hold if we use the expression of the unfolding to define a global map.
\medskip

A very interesting problem in Dynamical Systems is to
describe the global dynamics with hypotheses based on local
properties of the system. The Markus-Yamabe Conjecture is an
example but not the only one. For instance,  Alarc\'on {\sl et al.}~\cite{Alarcon}
prove the existence of a global attractor arising from a
unique local attractor, using the theory of free homeomorphisms of
the plane. Recently, Ortega and Ruiz del Portal in \cite{ORN}, have studied the global
behavior of an orientation preserving homeomorphism introducing
techniques based on the theory of prime ends. They
define the rotation number for some orientation preserving
homeomorphisms of $\R^2$ and show how this number gives information
about the global dynamics of the system. 
In this context, even a list of elementary concepts would be too long to include here. The discussion that follows may be taken as an appetizer for the reader willing to look them up properly in \cite{ORN}, \cite{Pom} and
\cite{Litt}.

The theory of prime ends was introduced by Carath\'eodory in order to
study the complicated shape of the boundary of a simply connected
open subset of $\R^2$.
When such a subset $U$ is non empty and
proper, by the Riemann mapping theorem, there is a conformal homeomorphism 
from $U$ onto the open unit disk.
Usually this homeomorphism cannot be extended to the closed disk.
Carath\'eodory's compactification associates the boundary of $U$
with the space of prime ends $\mathbb{P}$, which is homeomorphic
to $\mathbb{S}^1$.  
In that way, $U\cup \mathbb{P}$ is homeomorphic
to the closed unit disk.
The correspondence between points in the boundary of $U$
and points in $\mathbb{P}$ may be both multi-valued and not  one to one, but
 if $f$ is an orientation preserving
homeomorphism with $f(U)=U$, then $f$ induces an orientation
preserving homeomorphism $\tilde{f}$ in $\mathbb{P}$. Since the
space of prime ends is homeomorphic to the unit circle, the
rotation number of $\tilde{f}$ is well defined and the rotation
number of $f$ is defined to be equal to the rotation number of
$\tilde{f}$.

The points in $\partial_{\mathbb{S}^2}U$, the boundary of $U$ in the 
one point compactification of the plane, that play an important
role in the dynamics are accessible points. A point $\alpha \in
\partial_{\mathbb{S}^2}U$ is {\em accessible} from $U$ if there exists
an arc $\xi$ such that $\alpha$ is an end point of $\xi$ and  $\xi
\setminus \{\alpha\} \subset U$. Then $\alpha$ determines a
prime end $p(\alpha)\in \mathbb{P}$, which may not be unique,
such that $\xi \setminus \{p\} \cup \{p(\alpha)\}$ is an arc in $U\cup
\mathbb{P}$.

Accessible points are dense in $\partial_{\mathbb{S}^2}U$, but for
instance, in the case of fractal boundaries there exist points
which are not accessible from $U$. On the contrary, when the
boundary is well behaved, for instance an embedded curve of
$\mathbb{R}^2$, accessible points define a unique prime end. That
means that accessible periodic points of $f$ are periodic points
of $\tilde{f}$ with the same period.  
Consequently the rotation
number of $f$ is $1$ divided by the period. See \cite{Pom} and
\cite{Litt} for more details and definitions.

\begin{proposition} The examples $F_n$ in Theorem \ref{exper} have rotation number $1/n$.
\end{proposition}

\begin{proof}
By construction of the maps in Theorem~\ref{exper}, the basin of
attraction of the origin $$ U_n=\bigcup_{j=0}^{n-1}
R_n^j\left(h_n(U)\cap S_{1,n}\right)$$ is invariant by the map
$F_n$ and is a non empty and proper simply connected open set.
Moreover, as the periodic point $P$ is hyperbolic, the boundary of
$U$ is an embedded curve of $\mathbb{R}^2$ in a neighborhood of
$P$. In addition, $P$ is an accessible point from $U_n$, thus the
rotation number of $F_n$ is $\frac{1}{n}$. 
\end{proof}

The fact that the symmetry forces the maps in  Theorem~\ref{exper}
to have a rational rotation number seems to point out at a connection
between symmetry and  rotation number. It raises the question: for
orientation preserving homeomorphisms of the plane with a non
global asymptotically stable fixed point, does $\Z_n-$equivariance
imply a rational rotation number?

The question  is relevant because the rotation number gives strong
information about the global dynamics of the system. For instance,
 consider a dissipative orientation preserving $\Z_n-$equivariant
homeomorphism $f$ of the plane with an asymptotically stable fixed
point $p$. If the question has an affirmative answer, then
Proposition $2$ of \cite{ORN} implies that $p$ is a
global attractor under $f$ if and only if $f$ has no other
periodic point.

\bigbreak

\paragraph{\bf Acknowledgements}
The research of all authors at Centro de Ma\-te\-m\'a\-ti\-ca da Universidade do Porto (CMUP)
 had financial support from
 the European Regional Development Fund through the programme COMPETE and
 from  the Portuguese Government through the Fun\-da\-\c c\~ao para
a Ci\^encia e a Tecnologia (FCT) under the project\\
 PEst-C/MAT/UI0144/2011.
B. Alarc\'on was also supported from Programa Nacional de Movilidad de Recursos Humanos of the Plan Nacional de I+D+I 2008-2011 of the Ministerio de Educaci\'on (Spain) and grant MICINN-08-MTM2008-06065 of the Ministerio de Ciencia e Innovaci\'on (Spain).

\newpage

\appendix
\section*{Appendix --- \ Unfolding Theory for $\Z_n$}

In order to better understand the singularity for Szlenk's $\Z_4$-equivariant map, we calculate its codimension and provide a universal unfolding. Some of the information below may be retrieved from the $D_4$ equivariant set-up described for instance in Golubitsky {\em et al.} \cite{golu2}.

Let $\inv(\Z_4)$ be the set of $\Z_4$-invariant function germs from the plane to the reals. This is a ring generated by the following Hilbert basis
\begin{equation} \label{invariants}
\inv(\Z_4) = \left< N=x^2+y^2, A=x^4+y^4-6x^2y^2, B=(x^2-y^2)xy \right>
\end{equation}
in the sense that every germ in  $\inv(\Z_4)$ can be written in the form $\phi(N,A,B)$ where $\phi$ is a smooth function of three variables.

The set of $\Z_4$-equivariant map germs is a module over the ring of invariants; it is denoted by $\Equiv(\Z_4)$ and generated by the following
\begin{equation} \label{equivariants}
\begin{array}{lcl}
X_1  =  (x,y);  &\qquad& X_3  =  (x(x^2-3y^2),y(y^2-3x^2)); \\
X_2=(-y,x); &\qquad&
X_4=(-y(y^2-3x^2), x(x^2-3y^2).
\end{array}
\end{equation}

Two map-germs, $g$ and $h$, are $\Z_4$-contact-equivalent if (see Mather \cite{Mather1968}, even though we follow the notation in \cite{golu2}, chapter XIV) there exists an invertible change of coordinates $x \mapsto X(x)$, fixing the origin and $\Z_4$-equivariant, and a matrix-valued germ $S(x)$ satisfying for all $\gamma \in \Z_4$
$$
S(\gamma x)\gamma = \gamma S(x),
$$
with $S(0)$ and $dX(0)$ in the same connected component as the identity in the space of linear maps of the plane, and such that 
$$
g(x)=S(x)h(X(x)).
$$
The set of matrices satisfying the $\Z_4$-equivariance described above is denoted and generated as follows
$$
\equivM(\Z_4) = \left< S_j; T_j=iS_j, \; \; j=1,\hdots 4 \right> ,
$$
with
\begin{eqnarray*}
& &T_ i=\left( \begin{array}{cc} 
0 & 1 \\ 
-1 & 0 \end{array} \right),  
S_1 = \left( \begin{array}{cc} 
1 & 0 \\ 
0 & 1 \end{array} \right),
S_2=  \left( \begin{array}{cc} 
x^2 & xy \\ 
xy & y^2 \end{array} \right), \\
& & S_3= \left( \begin{array}{cc} 
-x^2 & xy \\ 
xy & -y^2 \end{array} \right),
S_4= \left( \begin{array}{cc} 
0 & x^3y \\ 
xy^3 & 0 \end{array} \right).
\end{eqnarray*}
Note that, in the $Z_4$-equivariant context, all map germs preserve the origin. In such cases as these, the tangent space $T$ to the $\Z_4$-contact orbit coincides with the restricted tangent space, $RT$.

The tangent space to $F_4$ is
$$
T_{\Equiv(\Z^4)}(F_4) = \left< (dF_4)X_i, S_jF_4, T_jF_4 \right>,
$$
where $X_i$ is one of the generators of $\Equiv(\Z_4)$ and $S_j$ and $T_j$ are the generators of $\equivM(\Z_4)$.

Given $F_4$ and dividing both components by $k$ as it does not affect the singularity, we have
$$
dF_4 = \left( \begin{array}{cc}
\frac{2xy^3}{(1+x^2+y^2)^2} & -\frac{3y^2(1+x^2+y^2)-2y^4}{(1+x^2+y^2)^2} \\
 & \\
\frac{3x^2(1+x^2+y^2)-2x^4}{(1+x^2+y^2)^2} & -\frac{2x^3y}{(1+x^2+y^2)^2} \end{array} \right).
$$
Note that all rows of this matrix have the common factor $1/(1+x^2+y^2)^2$, which does not affect the singularity. Also, all the products with $F_4$ will exhibit the common factor $1/(1+x^2+y^2)$, which again does not affect the singularity. We therefore present the generators of $T_{\Equiv(\Z^4)}(F_4)$ after a multiplication by the corresponding common factor. To exemplify,
$$
S_1 F_4 = (-\frac{y^3}{1+x^2+y^2}, \frac{x^3}{1+x^2+y^2})
$$
is reported as $S_1F_4=(-y^3,x^3)$.
This stated, we have the following list of generators of $T_{\Equiv(\Z^4)}(F_4)$, where the symbol $\sim$ indicates that a simplification was made through a product by a non-zero invariant:
\begin{eqnarray*}
(dF_4)X_1 & = & 3N(N-1)X_2+(N-1)X_4 \sim 3NX_2 + X_4; \\
(dF_4)X_2 & = & \frac{1}{4}(N(N+1)X_1-(N+1)X_3) \sim NX_1-X_3 \\
(dF_4)X_3 & = & \frac{3}{4}[(N^3+N^2+2A)X_2+(N^2+N-\frac{2}{3}A)X_4]; \\
(dF_4)X_4 & = & \frac{1}{4}[(N^3+6A+3N^2)X_1+(2A-3N^2-9N)X_3]; \\
S_1F_4 & = & 3NX_2+X_4 \\
S_2F_4 & = & -3BX_1-AX_2+NX_4 \\
S_3F_4 & = & \frac{1}{4}(NX_4-N^2X_2) \sim N^2X_2-NX_4 \\
S_4F_4 & = & (-\frac{1}{16}N^3-\frac{5}{32}NA)X_2+\frac{1}{8}BX_3+\frac{7}{32}N^2X_4 \\
T_1F_4 & = & \frac{1}{4}(3NX_1+X_3) \sim 3NX_1+X_3 \\
T_2F_4 & = & -BX_2; \\
T_3F_4 & = & \frac{1}{4}(A-N^2)X_1-BX_2; \\
T_4F_4 & = & \frac{1}{16}[(NA-N^3)X_1-14NBX_2-2BX_4].
\end{eqnarray*}

We use a filtration by degree $\mathcal{F}= \{ E^j\}_{j \in \N_0}$ of $\Equiv(\Z^4)$ where
$E^{j}\backslash E^{j+1}$ is the set of germs in $\Equiv(\Z^4)$ with all coordinates 
homogeneous polynomials of the same degree $j$ and
$E^0 = \Equiv(\Z^4)$. 
Note that 
$E^{2j} = E^{2j+1}$ for all $j \geq 0$ and each $E^j$ is a finitely generated $\inv(\Z_4)$-module.
Moreover, denoting as $\mathcal{M}(\Z^4)$ the unique maximal ideal in $\Equiv(\Z^4)$, we have
$$
\mathcal{M}(\Z^4) . E^j \subset E^{j+1}.
$$
We show that $E^5 \subset T_{\Equiv(\Z^4)}(F_4)$ by showing that
$$
E^5 \subset T_{\Equiv(\Z^4)}(F_4) +\mathcal{M}(\Z^4) E^5
$$
and invoking Nakayama's Lemma. We have that $E^5$ is generated over $\inv(\Z_4)$ as
\begin{equation}\label{generators}
E^5  = \left< N^2X_i, AX_i, BX_i, NX_j, AX_j, BX_j \right>, \; \; i=1,2; \; j=3,4.
\end{equation}
We point out that there are no equivariants of degree $6$ and therefore $E^6$ contains germs of degree $7$ or higher.

Multiply by $N$ the lower order generators of $T_{\Equiv(\Z^4)}(F_4)$, that is, $(dF_4)X_1$, $(dF_4)X_2$, $S_1F_4$ and $T_1F_4$ and append $AS_1F_4$ at the end of the list;
add or subtract as necessary terms in $\mathcal{M}(\Z^4) E^5$
to the generators of $T_{\Equiv(\Z^4)}(F_4)$. After performing these two operations, we obtain the matrix $Q$ below, where the entry $(i,j)$ is the coefficient of generator $j$ in (\ref{generators}) coming from the term $i$ in the list of generators of $T_{\Equiv(\Z^4)}(F_4)$:
$$
Q = \left( \begin{array}{cccccccccccc}
0 & 0 & 0 & 3 & 0 & 0 & 0 & 0 & 0 & 1 & 0 & 0 \\
1 & 0 & 0 & 0 & 0 & 0 &-1 & 0 & 0 & 0 & 0 & 0 \\
0 & 0 & 0 & 1 & 2 & 0 & 0 & 0 & 0 &1& -2/3& 0 \\
3 & 6 & 0 & 0 & 0 & 0 & -9 & 2 & 0 & 0 & 0 & 0 \\
0 & 0 & 0 & 3 & 0 & 0 & 0 & 0 & 0 & 1 & 0 & 0 \\
0 & 0 & -3 & 0 & -1 & 0 & 0 & 0 & 0 &1 & 0 & 0 \\
0 & 0 & 0 & 1 & 0 & 0 & 0 & 0 & 0 & -1 & 0 & 0 \\
0 & 0 & 0 & 0 & 0 & 0 & 0 & 0 &1/8 & 0 & 0 & 0\\
3 & 0 & 0 & 0 & 0 & 0 & 1 & 0 & 0 & 0 & 0 & 0 \\
 0 & 0 & 0 & 0 & 0 & -1 &  0 & 0 & 0 & 0 & 0 & 0 \\
-1/4 & 1/4  & 0 & 0 & -1 &  0 & 0 & 0 & 0 & 0 & 0 &0\\
0 & 0 & 0 & 0 & 0 & 0 & 0 & 0 & 0 & 0 & 0 & -2\\
0 & 0 & 0 & 0 & 0 & 0 & 0 & 0 & 0 & 0 & 1 & 0
\end{array} \right)
$$
The matrix $Q$ is of rank $12$, establishing our claim that $E^5 \subset T_{\Equiv(\Z^4)}(F_4)$.

We can then simplify the generators of $T_{\Equiv(\Z^4)}(F_4)$ even further adding the elements in 
$T_{\Equiv(\Z^4)}(F_4)\cap E^3\backslash E^5$:
$$
NX_1, X_3, 3NX_2+X_4.
$$
It is easily seen that there are the following two choices for a complement to $T_{\Equiv(\Z^4)}(F_4)$ inside $\Equiv(\Z^4)$
$$
V_1 = \{ X_1,  X_2, X_4 \} \quad \mbox{  and   } \quad V_2 = \{ X_1,  X_2, NX_2 \}.
$$
Therefore, the $\Z_4$-equivariant codimension of $F_4$ is $3$. A universal unfolding is given by
$$
G_4(x,y,\alpha, \beta, \delta) = F_4(x,y) + \alpha X_1 + \beta X_2 + \delta NX_2.
$$
Of course a choice using $V_1$ as a complement is just as good from the point of view of singularity theory. However, our choice yields better results for the construction of an example with symmetry $\Z_n$.

\begin{thebibliography}{ZZZZ}
\bibitem{Alarcon} B. Alarc\'on, V. Gu\'{\i}\~nez and C. Gutierrez. Planar Embeddings with a globally
attracting fixed point. {\em Nonlinear Anal.}, 69:(1), 140-150,
2008.

\bibitem{Alarcon-orbitas-periodicas} B. Alarc\'on, C. Gutierrez and J. Mart\'{\i}nez-Alfaro.
Planar maps whose second iterate has a unique fixed point. {\em J.
Difference Equ. Appl.}, 14:(4), 421-428, 2008.

\bibitem{ACL_global} B. Alarc\'on, S.B.S.D. Castro and I.S. Labouriau, Global Dynamics for Symmetric Planar Maps, Preprint CMUP 2012-12 (http://cmup.fc.up.pt/cmup/v2/frames/publications.htm)

\bibitem{BL} J. Bernat and J. Llibre, Counterexample to Kalman and Markus-Yamabe conjectures in dimension 4, Discrete of Continuous, Discrete and Impulsive Systems, 2, 337-379, 1996.

\bibitem{Litt} Cartwright, M. L., Littlewood, J. E.: Some fixed point theorems. Ann. of Math. 54, 1-37 (1951)

\bibitem{CEGHM} A. Cima, A. van den Essen, A. Gasull, E.-M. G. M. Hubbers and F. Ma\~nosas, A polynomial counterexample to the Markus-Yamabe conjecture, Advances in Mathematics, 131 (2), 453- 457, 1997.

\bibitem{Cima-Manosa} A. Cima, A. Gasull and F. Ma\~nosas. The Discrete
Markus-Yamabe Problem, {\em Nonlinear Analysis}, 35, 343-354, 1999.

\bibitem{F} R. Fessler, A solution to the two dimensional Global Asymptotic Jacobian Stability Conjecture, Annales Polonici Mathematici, 62, 45-75, 1995.


\bibitem{golu2} M. Golubitsky, I. Stewart and D.G. Schaeffer. Singularities and Groups in Bifurcation Theory Vol. 2. {\em Applied
Mathematical Sciences}, 69, Springer Verlag, 1985.

\bibitem{G} C. Gutierrez, A solution to the bidimensional Global Asymptotic Stability Conjecture, Ann. Inst. H. Poincar\'e. Anal. Non Lin\'eaire, 12 (6), 627-672, 1995.

\bibitem{L} J. LaSalle and S. Lefschetz, Stability by Liapunov's direct method with applications, Academic Press, 1961.

\bibitem{LS}  J. LaSalle, The stability of dynamical systems, CBMS-NSF Regional Conference Series in Applied Math., vol. 25, 1976.

\bibitem{MY} L. Markus and H. Yamabe, Global stability criteria for differential systems, Osaka Math. Journal, 12, 305-317, 1960.

\bibitem{Mather1968} J. Mather, Stability of $C^{\infty}$ maps, III. Finitely determined map-germs, Publ. Math. IHES, 35,127-156, 1968.

\bibitem{OrtegaBook} R. Ortega. Topology of the plane and periodic differential equations. Available at http://www.ugr.es/$\sim$ecuadif/fuentenueva.htm

\bibitem{ORN} R. Ortega and F. R. Ruiz del Portal, Attractors with vanishing rotation number, J. European Math. Soc., 13(6), 1569-1590, 2011.

\bibitem{Pom} Ch. Pommerenke, Boundary Behaviour of Conformal Maps. Grundlehren Math. Wiss. 299,
Springer (1991).

\end{thebibliography}
\end{document}